\documentclass[a4paper]{article}
\usepackage{mathtools}
\mathtoolsset{showonlyrefs}
\usepackage[all,cmtip]{xy}
\usepackage{amsmath}
\usepackage{amsthm}
\usepackage{amssymb}
\usepackage{amsxtra}
\usepackage{fontenc}
\usepackage[latin1]{inputenc}
\usepackage{graphicx}
\usepackage{multirow}
\usepackage[hang,small,labelfont=bf]{caption}
\usepackage{pstricks}
\usepackage[ruled]{algorithm2e}
\usepackage{etex}

\SetKw{KwInit}{Initialization:}
\SetKw{KwBreak}{break}
\SetKw{KwInput}{Input:}
\SetKw{KwOutput}{Output:}
\SetKw{KwIter}{Iteration:}

\newcommand{\dx}[1][x]{\,\mathrm{d}#1}

\def\argmin{\mathop{\rm argmin}}

\def\div{\mathop{\rm div}}

\newcommand{\R}{\ensuremath{\mathbb{R}}}

\newcommand{\abs}[1]{\ensuremath{\left\vert#1\right\vert}}

\newcommand{\XX}{\mathcal{X}}
\newcommand{\YY}{\mathcal{Y}}

\let\abs=\envert

\newcommand{\norm}[2]{\left\| #1 \right\|_{#2}}

\def\argmin{\mathop{\rm argmin}}
\newcommand{\tT}{{\scriptscriptstyle \operatorname{T}}}

\DeclareMathOperator{\Id}{Id}

\DeclareMathOperator{\diag}{diag} 

\DeclareMathOperator{\supp}{supp}

\newtheorem{theorem}{Theorem}
\newtheorem{remark}{Remark}
\newtheorem{lemma}{Lemma}

\begin{document}
\author{
Jan Henrik Fitschen, Friederike Laus and Gabriele Steidl \\
\small Department of Mathematics,
 University of Kaiserslautern, Germany \\ \small
  \{fitschen, friederike.laus, steidl\}@mathematik.uni-kl.de}

\title{ 
Dynamic Optimal Transport with
Mixed Boundary Conditions for
Color Image Processing
}
\date{\today}
\maketitle

\begin{abstract}
 \noindent
Recently, Papadakis et al.~\cite{PPO14} proposed an efficient primal-dual algorithm for solving the dynamic optimal transport 
problem with quadratic ground cost and measures having densities with respect to the Lebesgue measure. 
It is based on the fluid mechanics formulation by Benamou and Brenier~\cite{BB00} and proximal splitting schemes. 
In this paper we extend the framework to color image processing. 
We show how the transportation problem for RGB color images can be tackled by prescribing periodic boundary conditions in the color dimension.
This requires the solution of a 4D Poisson equation with mixed Neumann and periodic boundary conditions 
in each iteration step of the algorithm. This 4D Poisson equation 
can be efficiently handled by fast Fourier and Cosine transforms.
Furthermore, we sketch how the same idea can be used in a modified way to transport periodic 1D data 
such as the histogram of cyclic hue components of images. 
We discuss the existence and uniqueness of a minimizer of the associated energy functional.
Numerical examples illustrate the meaningfulness of our approach.
\end{abstract}
%
\section{Introduction} \label{sec:intro}
Recently, methods from optimal transport (OT) have gained a lot of interest in image processing. 
In one of the first applications, the Wasserstein distance (earth mover distance) 
has been successfully used for image retrieval~\cite{RTB00} and since then it has been applied 
to many other tasks as 
color transfer~\cite{PK07, RP11},
(co)segmentation~\cite{NBCE09,PFR12,SS13}, the synthesis and mixing of stationary Gaussian textures~\cite{XFPC13} and the computation of barycenters~\cite{RPDB13,MRSS14}.

The basic problem, going back to Monge (1746-1818), can be formulated as follows: 
Given two probability spaces $(\XX,\mu_0)$ and $(\YY,\mu_1)$ and a nonnegative cost function $c(x,y)$ 
on $\XX\times\YY$, find a transport map $T\colon \XX\to\YY$ that transports the mass of $\mu_0$ to the mass of $\mu_1$ at minimal cost, i.e.,
$T$ minimizes
\begin{equation}\label{Monge_problem}
\int_{\XX} c\bigl(x,T(x)\bigr) \dx[\mu_0(x)] \quad \mbox{subject to} \quad \mu_0\circ T^{-1} = \mu_1.
\end{equation}
One of the major limitations in applications using OT is the fact that it is in general not known whether a solution of problem~\eqref{Monge_problem} 
exists and even in this case the computation of the optimal map $T$ is usually a demanding task (except very few cases, e.g.\ OT on $\R$ with convex costs). 
We focus in the following on a specific instance of problem~\eqref{Monge_problem}, 
namely if $\XX =\YY = \R^{d}$, $c(x,y) = \frac{1}{2}\abs{x-y}^{2}$ and $\mu_0$ and $\mu_1$ are absolutely continuous w.r.t.\ the Lebesgue measure, 
that means there exist probability density functions $f^0$ and $f^1$ with
\begin{equation*}
  \mu_i(A) = \int_{A} f^i(x)\dx,\quad i = 0,1,\quad A\in \mathcal{B}(\R^{d}).
\end{equation*}
In this case there exists a unique optimal transport map $T$ that transports $f^0$ to $f^1$,
see, e.g., \cite{GM96,Vil03}. 
Instead of considering a time independent, ``static'' mass transportation problem 
one may alternatively consider the geodesic path between the two measures w.r.t.\ 
to the Wasserstein metric (the so called displacement interpolation~\cite{McC97}). 
While the static problem can be seen as a distance problem (find the minimal distance between the probability measures $\mu_0$ and $\mu_1$), 
the dynamic problem can be interpreted as a geodesic problem (find an optimal path between $\mu_0$ and $\mu_1$). 
For the $L^2$ ground cost this geodesic is obtained by linear interpolation between the identity and the optimal transport map $T$,
i.e., $\mu_t = \mu_0 \circ T_t^{-1}$, where $T_t = (1-t)\Id + tT$.
Benamou and Brenier \cite{BB00} gave the following equivalent formulation of the dynamic OT problem in terms of fluid mechanics:
minimize
\begin{equation}
\int_{[0,1]} \int_{\R^d}  \frac{1}{2} f(x,t)\abs{v(x,t)}^2 \, \text{d}x \ \text{d}t,\label{Monge_dynamic}
\end{equation}
subject to
$\bigcup_{t\in [0,1]} \supp{f(\cdot,t)}$ bounded, $f(\cdot,0) = f^0, f(\cdot,1) = f^1$
and $\partial_t f + \div_x (fv) = 0$, $(f,v)$ sufficiently smooth.
Substituting $m = fv$, this problem becomes convex and can be treated by respective algorithms.

In addition to the Dirichlet boundary condition for the time interval, 
problem~\eqref{Monge_dynamic} needs  
to be  equipped with spatial boundary conditions in practical applications
(appearing in the momentum variable $m$). 
A natural choice which was also used in \cite{PPO14} for  gray-value images are Neumann boundary conditions. 
In this paper we want to deal with color  RGB (red, green, blue)  images. 
More precisely, we consider a $m \times n$ RGB image as 
a 3D  object of size $m \times n \times 3$ with the color values in the third dimension, 
i.e., we interpret these images as (realization of) a 3D density function. 
To use Neumann boundary conditions for the color dimension is certainly not a good idea
since the solution can depend on the ordering of the color channels.
Therefore, we suggest to establish periodic boundary conditions in the third dimension. 
Note that we have Dirichlet boundary conditions in time, and  Neumann plus periodic spatial boundary conditions. 
Fig.~\ref{Fig:boundary} illustrates the effect of the different boundary conditions. 
%
\begin{figure*} \centering
{\includegraphics[width=.1\textwidth]{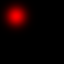}} 
{\includegraphics[width=.1\textwidth]{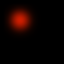}} 
{\includegraphics[width=.1\textwidth]{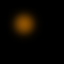}} 
{\includegraphics[width=.1\textwidth]{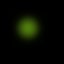}} 
{\includegraphics[width=.1\textwidth]{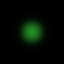}} 
{\includegraphics[width=.1\textwidth]{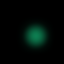}} 
{\includegraphics[width=.1\textwidth]{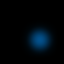}} 
{\includegraphics[width=.1\textwidth]{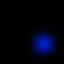}} 
{\includegraphics[width=.1\textwidth]{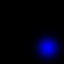}} \\[1.5ex]
{\includegraphics[width=.1\textwidth]{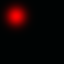}} 
{\includegraphics[width=.1\textwidth]{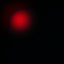}} 
{\includegraphics[width=.1\textwidth]{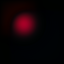}} 
{\includegraphics[width=.1\textwidth]{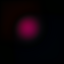}} 
{\includegraphics[width=.1\textwidth]{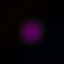}} 
{\includegraphics[width=.1\textwidth]{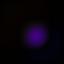}} 
{\includegraphics[width=.1\textwidth]{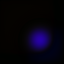}} 
{\includegraphics[width=.1\textwidth]{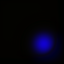}} 
{\includegraphics[width=.1\textwidth]{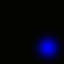}}
\caption{\label{Fig:boundary}
OT results for transporting a red Gaussian into a blue one with different boundary conditions
for the third (color) dimension. 
First row: Neumann boundary conditions; the color is transported from red over green to blue.
This changes if we permute the RGB channels.
Second row:  periodic boundary conditions; the color is transported from red over violet to blue which is more intuitive
and does not change for permuted color channels.
}
\end{figure*}
%
In Section~\ref{sec:model} we explain the discretization of problem~\eqref{Monge_dynamic} 
emphasizing the mixed boundary conditions. In particular we deal with the existence and uniqueness of minimizers.
Similarly as suggested in \cite{PPO14} we apply a primal dual algorithm to find a minimizer in Section \ref{sec:alg}.
In contrast to \cite{PPO14} each iteration step of this algorithm requires the solution of
a 4D Poisson problem with mixed Neumann and periodic boundary conditions which can be efficiently computed by applying FFTs
and fast cosine transforms.
In Section~\ref{sec:RGB} we apply our findings to the dynamic OT of RGB images which was the initial motivation of this work. 
Another application is given in Section \ref{sec:hue}.
Here, the (cyclic) OT is applied to HSV (hue, saturation, value) images, where only the cyclic hue component
is transported. For further details we refer to \cite{Laus15}.
%
\section{Model for dynamic OT with mixed boundaries} \label{sec:model}  
Rewriting~\eqref{Monge_dynamic} for $m = fv$, we see that the geodesic path 
between measures with probability densities $f^0 = f(\cdot,0)$ and $f^1= f(\cdot,1)$ has density $f(\cdot,t)$ fulfilling
\begin{equation} \label{cont_model}
\argmin_{(m,f) \in {\cal C} }  \int_{[0,1]}\int_{[0,1]^d} J(m,f) \, \text{d}x \ \text{d}t,
\end{equation}
where
\begin{align}
J(&m(x,t),f(x,t)) 
:= 
\left\{
\begin{array}{cl}
\frac{|m(x,t)|^2}{2 f(x,t)} & {\rm if} \; f(x,t) > 0,\\
0 & {\rm if} \; (m(x,t), f(x,t)) =  \left( 0_d^\tT,0 \right),\\
+ \infty & {\rm otherwise},
\end{array}
\right.
\end{align}
\begin{align} \label{cont_bound}
\mathcal{C} := 
\big\{
&(f,m): 
\partial_t f + {\rm div}_x m = 0,
f(\cdot,0) = f^0, \; f(\cdot,1) = f^1\big\}
\end{align}
with appropriate boundary conditions.
In the following, we describe the discretization of problem \eqref{cont_model} for one spatial dimension with
cyclic spatial boundary conditions. 
The discretization of the continuity equation demands the evaluation of discrete partial derivatives in time as well as in space. In order to avoid solutions suffering from the well-known checkerboard-effect (see for instance~\cite{Pat80}) we adopt the idea of a staggered grid as in \cite{PPO14}.

{\bf Discretization (Spatial 1D):}
We consider the values of $f$  at spatial cell midpoints 
$\frac{j-1/2}{n}$, $j=1,\ldots,n$
and time 
$\frac{k}{p}$, $k=1,\ldots,p-1$, and denote the corresponding array by
$\left( f(j-\frac12,k) \right)_{j=1,k=1}^{n,p-1} \in \R^{n,p-1}$. 
The boundary values are assumed to be fixed in
$f^0 := \left( f(\frac{j-1/2}{n},0) \right)_{j=1}^n$ 
and 
$f^1 := \left( f(\frac{j-1/2}{n},1) \right)_{j=1}^n$,
where $f^i \ge 0$, $i=0,1$ and
$\|f^0\|_1 =  \|f^1\|_1$. We can skip the normalization $\|f^0\|_1 = 1$ here.
The values of $m$ are taken at the cell faces 
$\frac{j}{n}$, $j = \kappa,\ldots,n-1$
and time 
$\frac{k-1/2}{p}$, $k=1,\ldots,p$ 
and we consider the array
$\left( m(j,k-\frac12) \right)_{j=\kappa,k=1}^{n-1,p} \in \R^{n-\kappa,p}$,
where $\kappa = 1$ for Neumann boundary conditions and $\kappa = 0$
for periodic boundary conditions, see Fig. \ref{fig:grid}.
To give a sound matrix-vector notation of the discrete minimization problem
we reorder $m$ and $f$ columnwise into vectors
$f = {\rm vec}(f) \in \R^{n(p-1)}$, 
$m  =  {\rm vec} (m) \in \R^{(n-\kappa)p}$.
Since it becomes clear from the context if we deal with arrays or vectors
we use the same notation.
The derivatives in $\mathcal{C}$
are approximated by forward differences 
and  the integral in \eqref{cont_model} by a midpoint rule, 
where the midpoints 
$\left( u(j-\frac12,k-\frac12) \right)_{j,k=1}^{n,p}$ and $\left( v(j-\frac12,k-\frac12)\right)_{j,k=1}^{n,p}$ 
are computed by averaging the neighboring two values.
%
\begin{figure}[ht]
\begin{center}
\includegraphics[width=0.25\textwidth]{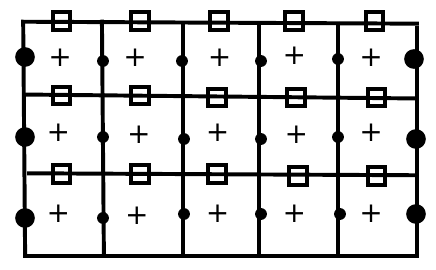}
\end{center}
\caption{ 
\label{fig:grid} Discretization grid  for the dynamic OT problem for 1D signals $f$:
horizontal direction  for time; vertical direction for space;
$\bullet$ given boundary sampling nodes $f^0$ and $f^1$,
$\cdot$ sampling nodes for $f((j-1/2)/n,k)$, $j = 1,\ldots,n$, $k=1,\ldots,p-1$,
$\square$ sampling nodes for $m$,
'+' quadrature nodes.
}

\end{figure}
%
This results in the following discrete model:
\begin{align} \label{model_disc}
&\argmin_{ m,f,u,v } \| J(u, v) \|_1 , \\
&\mbox{subject to} \quad  S_{M} m = u, \quad S_{F} f + f_b^+ = v, \label{subst}\\
&\qquad \qquad \quad \underbrace{( D_{M} |  D_{F} )}_{A}
\begin{pmatrix} m\\f \end{pmatrix} = f_b^-. \label{cont_d}
\end{align}
We denote 
\begin{equation} \label{bound_disc}
{\cal C}_d := \left\{
\begin{pmatrix} m\\f \end{pmatrix} \in \R^{(n-\kappa)p+n(p-1)}: 
( D_M |  D_F )
\begin{pmatrix} m\\f \end{pmatrix} = f_{b}^- 
\right\}.
\end{equation}
The involved vectors are defined as
\begin{align}
f_{b}^+ &:= \frac12 \bigl( (f^0)^\tT, 0_{n(p-2)},  (f^1)^\tT \bigr)^\tT
, \\
f_{b}^- &:= p \bigl( (f^0)^\tT, 0_{n(p-2)}, -(f^1)^\tT \bigr)^\tT
\end{align}
and the matrices using the Kronecker product $\otimes$ as

\begin{align}
S_F := S_p^ \tT \otimes I_n, \quad D_F := -D_p^\tT \otimes I_n,
\end{align}
\begin{align}
S_M := \left\{
\begin{array}{ll}
I_p \otimes S_n^\tT & \: \mbox{Neumann},\\
I_p \otimes S_{n,per}^\tT & \: \mbox{periodic},
\end{array}
\right.
D_M :=
\left\{
\begin{array}{ll}
I_p \otimes -D_n^\tT & \mbox{Neumann},\\
I_p \otimes D_{n,per}^\tT & \mbox{periodic},
\end{array}
\right.
\end{align}
and
{\scriptsize
$$
S_{n,per} := 
\frac12 \left( 
\begin{array}{rrrrrrr}
1&0& &      & & &1\\
1&1& &            \\
 & & &\ddots& & &   \\
 & & &      &1&1&0\\
 & & &      & &1&1
\end{array}
\right) \in \R^{n,n},
$$
$$
S_p := 
\frac12 \left( 
\begin{array}{rrrrrr}
1&1&              \\
 &1&1&            \\
 & & &\ddots& &   \\
 & & &      &1&1
\end{array}
\right) \in \R^{p-1,p},
$$
$$
D_{n,per} := 
n \left( 
\begin{array}{rrrrrr}
-1&  &   &   &  &1 \\
 1&-1&          \\
  & & &\ddots& & \\
  & & &      &-1& 0\\
  & & &      &1&-1
\end{array}
\right) \in \R^{n,n},
$$
$$
D_p := p \left( 
\begin{array}{rrrrrrrr}
-1&\; 1&                 \\
  &-1  & \; 1&                \\
  &    &     &\ddots&  &    \\
  &    &     &      &  &-1&1& 0\\
  &    &     &      &  &  &-1& 1
\end{array} 
\right) \in \R^{p-1,p}.
$$
}
Since we have to be slightly careful concerning the uniqueness of the solution
in the periodic setting we provide the following proposition.
%
\begin{theorem} \label{existence}
The discrete dynamic transport model \eqref{model_disc} has a solution.
\end{theorem}
%
\begin{proof}
For periodic boundary conditions and even $n$, we have ${\cal N}(S_M)= \{w \otimes \tilde 1_n: w \in \mathbb R^p\}$ 
with $\tilde 1_n := (1,-1,\ldots,1,-1)^\tT \in \mathbb R^n$,
and ${\cal N}(S_M) = \{0_{np}\}$ otherwise.

The constraints in ${\cal C}_d$ can be rewritten as
$$
 D_F f = f_b^- - D_Mm,\quad
 f= D_F^\dagger f_b^- - D_F^\dagger D_M m
$$
with the Moore-Penrose inverse $D_F^\dagger = (D_F^\tT D_F)^{-1} D_F^\tT$.
Then we obtain
$$
\argmin_{u,v} \|J(u,v)\|_1 = \argmin_{m} \left \| J \left( X(m),Y(m) \right) \right\|_1 
$$
with 
$X(m) := S_M m$, $Y(m) := - S_F D_F^\dagger D_M m + S_F D_F^\dagger f_b^- + f_b^+$.
Let $\|m\|_2 \rightarrow +\infty$ and assume that 
$ \| J \left( X(m),Y(m) \right) \|_1$ is bounded.
Then each of the quotients is bounded, i.e., there exists $c>0$ such that
$X(m)_i^2 \le c |Y(m)_i|$ and therefore
$\|X(m)\|_2^2 \le c \|Y(m)\|_1$.
Thus, in the case ${\cal N}(S_M) = \{0_{np}\}$, we get
\begin{align} \label{exist_1}
(1/\|S_M^\dagger\|_2 ^2) \|m\|_2^2 &\le \|X(m)\|_2^2 \le c \|Y(m)\|_1 \\
&\le c \| S_F D_F^\dagger D_M\|_1 \|m\|_1 + \tilde c,
\end{align}
which is not possible as $\|m\|_2 \rightarrow +\infty$.
In the case ${\cal N}(S_M)= \{w \otimes \tilde 1_n: w \in \mathbb R^p\}$ 
we use the orthogonal splitting $m = m_R + w \otimes \tilde 1_n$, where $m_R \in {\cal R}(S_M^\tT)$.
Straightforward computation shows that 
$S_F D_F^\dagger D_M ( w \otimes \tilde 1_n) =  \tilde w \otimes \tilde 1_n$
with some
$\tilde w \in \mathbb R^p$.
Since $Y(m) \ge 0$, we conclude
\begin{equation} \label{another_norm}
\|- S_F D_F^\dagger D_M m_R + S_F D_F^\dagger f_b^- + f_b^+\|_\infty \ge \| \tilde w \otimes \tilde 1_n \|_\infty.
\end{equation}
If $m_R$ remains finite, then $\tilde{1}_N \otimes \widetilde{w}$ remains finite as well.
	Since the kernel of $S_{\rm f} D^\dagger_{\rm f} D_{\rm m}$ consists of constant vectors, 
	this is only possible if $w$ is a multiple of $1_P$.
	But in this case $J_p(X(m),Y(m))$ has a finite value which is reached.
	It remains to consider the case $\|m_R\|_2 \to +\infty$.
Then we obtain similarly as in \eqref{exist_1} that
\begin{align}
(1/\|S_M^\dagger|_{{\cal R} (S_M^\tT) }\|_2 ^2) \|m_R\|_2^2 
\le 
\|X(m_R)\|_2^2 \le c \|Y(m)\|_1 \\
\le c\|- S_F D_F^\dagger D_M m_R + S_F D_F^\dagger f_b^- + f_b^+\|_1 + c\| \tilde w \otimes \tilde 1_n \|_1.
\end{align}
By \eqref{another_norm} we see that this is not possible as $\|m_R\|_2 \rightarrow +\infty$.
In summary, we have that
$ \| J \left( X(m),Y(m) \right) \|_1$ is coercive
and since it is also proper and lower semi-continuous, it has a minimizer.
\end{proof}
Unfortunately, 
$J(m,f)$ is not strongly convex on its domain.
As it can be seen in the following lemma it is even not strictly convex.
%
\begin{lemma} \label{perspective}
For any two minimizers $(m_i,f_i)$, $i=1,2$ the relation
$$\frac{S_M m_1}{S_F f_1 + f_b^-} = \frac{S_M m_2}{S_F f_2 + f_b^-}$$
holds true.
\end{lemma}
%
\begin{proof} 
The function $J(u,v)$ is the perspective function of the strictly convex 
function $\psi = | \cdot |^2$,
i.e.,  $J(u,v) = v \psi\left(\frac{u}{v} \right)$, see, e.g., \cite{DM08}.
For $\lambda \in (0,1)$ and $(u_i,v_i)$ with $v_i >0$, $i=1,2$, we have (componentwise)
\begin{align}
 J \left( \lambda (u_1,v_1) + (1-\lambda) (u_2,v_2) \right)
 = 
 \left( \lambda v_1 + (1-\lambda)v_2 \right) \psi \left( \frac{\lambda u_1 + (1-\lambda)u_2}{\lambda v_1 + (1-\lambda)v_2} \right)\\
= 
(\lambda v_1 + (1-\lambda) v_2) \psi 
\left( 
\frac{\lambda v_1}{\lambda v_1 + (1-\lambda)v_2} \frac{u_1}{v_1}\right. 
\left.
+ \frac{(1-\lambda) v_2}{\lambda v_1 + (1-\lambda)v_2} \frac{u_2}{v_2}
\right) 
\end{align}
and if 
$\frac{u_1}{v_1} \not = \frac{u_2}{v_2}$
by the strict convexity of $\psi$,
\begin{align}
J \left( \lambda (u_1,v_1) + (1-\lambda) (u_2,v_2) \right)
< 
\lambda J(u_1,v_1) + (1-\lambda)J(u_2,v_2),
\end{align}
which proves the assertion.
\end{proof}
%
\begin{remark} \label{unique}
For periodic boundary conditions, even $n$ and 
$f^1 = f^0 + \gamma \tilde 1_n$,
$\gamma \in [0,\min f^0 )$ the minimizer of \eqref{model_disc} is not unique which can be seen as follows:
obviously, we would have a minimizer $(m,f)$ if 
$m = w \otimes \tilde 1_n \in {\cal N}(S_M)$ for some $w \in \mathbb R^p$
and there exits $f\ge 0$ which fulfills the constraints.
Setting $f^{k/p} := f(j-1/2,k)_{j=1}^n$, $k=0,\ldots,p$, these constraints read
$-2p w \otimes \tilde 1_n = p (f^{(k-1)/p} - f^{k/p})_{k=1}^p$. 
Thus, any $w \in \mathbb R^p$ such that
\begin{align} 
	f^{1/p} &= f^0 + 2w_1 \tilde1_n, \;
	f^{2/p} = f^0 + 2(w_1 + w_2) \tilde 1_n, \ldots \, ,\\
	f^{1} &= f^0 + 2(w_1 + w_2+ \ldots+w_p) \tilde 1_n
\end{align}
are nonnegative vectors provides a minimizer of \eqref{model_disc}.
We conjecture that the solution is unique in all other cases,
but have no proof so far.
\end{remark}
%
\section{Minimization Algorithm} \label{sec:alg}  
We apply the primal-dual algorithm of Chambolle and Pock \cite{CP11}
in the form of Algorithm 8 in \cite{BSS14}.
%
	\begin{algorithm}
	{\small	\KwInit{$m^{(0)}=0_{np}$, $f^{(0)}=0_{n(p-1)}$, 
		$b_m^{(0)}=b_f^{(0)}=\bar b_m^{(0)}=\bar b_f^{(0)}=0_{np}$, 
		$\theta =1$, $\tau,\sigma$ with  $\tau \sigma < 1$.} \\
		\KwIter{For $r = 0,1,\ldots$ iterate}
		\begin{align}
			1. \begin{pmatrix} m^{(r+1)} \\ f^{(r+1)} \end{pmatrix}  
			&:= \ \argmin_{(m,f) \in {\cal C}_d} 
			 \frac{1}{2\tau} \| 
			\begin{pmatrix} m \\ f \end{pmatrix} 
			- \begin{pmatrix} m^{(r)} \\ f^{(r)} \end{pmatrix} +
			\tau \sigma  
			\begin{pmatrix} 
			S_M^\tT & 0 \\
			0             & S_F^\tT
			\end{pmatrix} 
			\begin{pmatrix} \bar b_m^{(r)} \\ \bar b_f^{(r)} \end{pmatrix} \|_2^2 
			 \\
			2. \ \begin{pmatrix} u^{(r+1)} \\ v^{(r+1)} \end{pmatrix}  
			&:= \ \argmin_{u,v} \|J(u,v) \|_1 
			 + \frac{\sigma}{2} 
			 \|
			\begin{pmatrix} u \\ v \end{pmatrix} \\
			 -& \begin{pmatrix} 
			S_M & 0 \\
			0             & S_F
			\end{pmatrix}
			\begin{pmatrix} m^{(r+1)} \\ f^{(r+1)} \end{pmatrix} 
			-
			\begin{pmatrix} 
			 0 \\
			f_b^+
			\end{pmatrix}
			- 
			\begin{pmatrix} b_m^{(r)} \\ b_f^{(r)} \end{pmatrix} 
			\|_2^2			
			\\
			3. \qquad 
			b_m^{(r+1)} &:= \ b_m^{(r)} + S_M m^{(r+1)} - u^{(r+1)} \\
			b_f^{(r+1)} &:= \ b_f^{(r)} + S_F f^{(r+1)} + f_b^+ - v^{(r+1)}    		
			\\ 
			4. \qquad 
			\bar b_m^{(r+1)} &:= \ b_m^{(r+1)} + \theta (b_m^{(r+1)}-b_m^{(r)} ) \\
			\bar b_f^{(r+1)} &:= \ b_f^{(r+1)} + \theta (b_f^{(r+1)}-b_f^{(r)} )
		\end{align}
		\caption{PDHG Algorithm for solving \protect \eqref{model_disc}}}%
	\end{algorithm} \label{alg1}
%
{\bf Step 1} requires the projection of 
$$
a := \begin{pmatrix} m^{(r)} \\ f^{(r)} \end{pmatrix} -
			\tau \sigma  
			\begin{pmatrix} 
			S_M^\tT & 0 \\
			0             & S_F^\tT
			\end{pmatrix} 
			\begin{pmatrix} \bar b_m^{(r)} \\ \bar b_f^{(r)} \end{pmatrix}			
$$
onto ${\cal C}_d$ which is given by 
\begin{align}
\Pi_{{\cal C}_d} (a) &= a - A^\tT (A A^\tT )^\dagger (Aa- f_b^-), \\  
(AA^\tT)^\dagger &=  Q \diag(\tilde \lambda_j) Q^\tT,
\end{align}
where $A A^\tT$ has the spectral decomposition 
$
A A^\tT = Q \diag(\lambda_j)  Q^\tT
$ 
and 
$\tilde \lambda_j := 1/\lambda_j$ if $\lambda_j >0$ and zero otherwise.
For $A$ in \eqref{cont_d} the application of $(AA^\tT)^\dagger$
requires the solution of a 2D Poisson equation. 
In case of periodic boundary conditions we get
\begin{align}
A A^\tT  &= I_p \otimes D_{n,per}^\tT D_{n,per} + D_{p}^\tT D_{p} \otimes I_n \\
&= I_p \otimes  n^2 \Delta_{n,per} + p^2 \Delta_p \otimes I_n,
\end{align}
where
{\scriptsize
\begin{align}
\Delta_{n,per} &:= 
\left(
\begin{array}{rrrrrrrr}
2 &-1&  &      &  &  &  &-1            \\
-1& 2&-1&                  \\
  &  &  &\ddots&  &        \\
  &  &  &      &  &-1&2 &-1 \\
-1&  &  &      &  &  &-1& 2
\end{array} 
\right), \\
\Delta_p &:= 
\left(
\begin{array}{rrrrrrrr}
1 &-1&  &      &  &  &  &0            \\
-1& 2&-1&                  \\
  &  &  &\ddots&  &        \\
  &  &  &      &  &-1&2 &-1 \\
0&  &  &      &  &  &-1& 1
\end{array} 
\right) .
\end{align}
}
Since 
$$
\Delta_{n,per} = \frac1n \bar F_n \diag(q_{n,per} ) F_n, \quad q_{n,per} := \left( 4\sin^2 \frac{j \pi}{n} \right)_{j=0}^{n-1}
$$ 
with the Fourier matrix
$F_n := (\mathrm{e}^{-2\pi \mathrm{i}jk})_{j,k=0}^{n-1}$
and
$$
\Delta_p = C_p^\tT \diag(q_p) C_p, \quad q_p := \left( 4\sin^2 \frac{j \pi}{2p} \right)_{j=0}^{p-1}
$$
with the DCT-II matrix
$$C_p := \sqrt{\frac{2}{p}} \left( \epsilon_j \cos\frac{j(2k+1)\pi}{2p} \right)_{j,k=0}^{p-1}, \quad 
\epsilon_j := 
\left\{
\begin{array}{ll}
1/\sqrt{2} &{\rm if} \; j=0,\\
1&{\rm otherwise}
\end{array} 
\right.
$$
we obtain 
\begin{align}
A A^\tT  = &(C_p^\tT \otimes \frac1n \bar F_n ) (n^2 I_p \otimes \diag(q_{n,per} ) \\ 
&+ p^2 (\diag(q_p) \otimes I_n) (C_p \otimes F_n),
\end{align}
so that its pseudo-inverse can be computed  by the FFT and the fast cosine transform.

\begin{remark} \label{poisson_4D}
For the transport of general 2D RGB images we have analogously to solve a 4D Poisson equation
with Neumann boundary conditions and a periodic boundary condition for the color channels.
\end{remark}

{\bf Step 2} of the algorithm can be computed componentwise as proposed in \cite{PPO14}.
Setting 
$a_m := S_M m^{(r+1)} + b_m^{(r)}$, $a_f := S_F f^{(r+1)} + b_f^{(r)} + f_b^+$
we have to find componentwise 
$$
\argmin_{u,v} \frac{u^2}{2v} + \frac{\sigma}{2} (u -a_m)^2 + \frac{\sigma}{2} (v -a_f)^2.
$$
Setting the gradient to zero yields
\begin{align}
\frac{u}{v} + \sigma (u - a_m) = 0,\quad
-\frac12 \frac{u^2}{v^2} + \sigma (v-a_f) =0.
\end{align}
Thus,
$$
u = \frac{\sigma v a_m}{1 + \sigma v}
$$
and $v$ is the solution of the third order equation
$$
f(v) = 2(1+\sigma v)^2 (v-a_f) - \sigma a_m^2 = 0.
$$
This can be solved by few Newton steps
which can be computed simultaneously for all components.
Alternatively, one may use  Cardan's formula.
\begin{figure*} \centering
{\includegraphics[width=.10\textwidth]{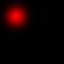}} 
{\includegraphics[width=.10\textwidth]{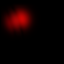}} 
{\includegraphics[width=.10\textwidth]{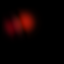}} 
{\includegraphics[width=.10\textwidth]{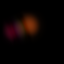}} 
{\includegraphics[width=.10\textwidth]{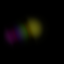}} 
{\includegraphics[width=.10\textwidth]{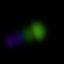}} 
{\includegraphics[width=.10\textwidth]{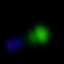}} 
{\includegraphics[width=.10\textwidth]{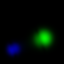}} 
{\includegraphics[width=.10\textwidth]{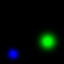}} \\[1.5ex]
{\includegraphics[width=.10\textwidth]{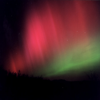}} 
{\includegraphics[width=.10\textwidth]{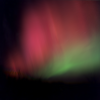}} 
{\includegraphics[width=.10\textwidth]{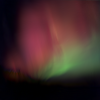}} 
{\includegraphics[width=.10\textwidth]{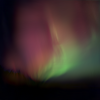}} 
{\includegraphics[width=.10\textwidth]{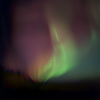}} 
{\includegraphics[width=.10\textwidth]{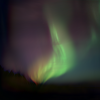}} 
{\includegraphics[width=.10\textwidth]{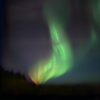}} 
{\includegraphics[width=.10\textwidth]{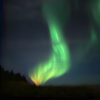}} 
{\includegraphics[width=.10\textwidth]{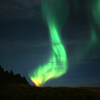}} \\[1.5ex]
{\includegraphics[width=.10\textwidth]{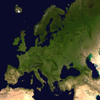}} 
{\includegraphics[width=.10\textwidth]{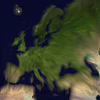}} 
{\includegraphics[width=.10\textwidth]{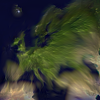}} 
{\includegraphics[width=.10\textwidth]{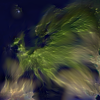}} 
{\includegraphics[width=.10\textwidth]{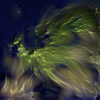}} 
{\includegraphics[width=.10\textwidth]{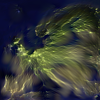}} 
{\includegraphics[width=.10\textwidth]{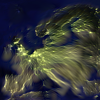}} 
{\includegraphics[width=.10\textwidth]{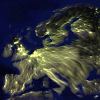}} 
{\includegraphics[width=.10\textwidth]{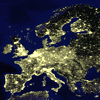}} \\[1.5ex]
{\includegraphics[width=.10\textwidth]{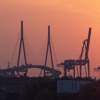}} 
{\includegraphics[width=.10\textwidth]{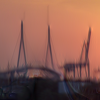}} 
{\includegraphics[width=.10\textwidth]{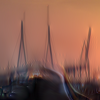}} 
{\includegraphics[width=.10\textwidth]{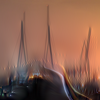}} 
{\includegraphics[width=.10\textwidth]{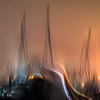}} 
{\includegraphics[width=.10\textwidth]{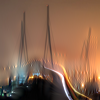}} 
{\includegraphics[width=.10\textwidth]{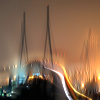}} 
{\includegraphics[width=.10\textwidth]{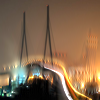}} 
{\includegraphics[width=.10\textwidth]{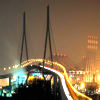}} 
\caption{\label{Fig:RGB_transport}
OT between RGB images. The images are displayed at intermediate time $t = \frac{i}{8}, i = 0,\dots, 8$.}
\end{figure*}

\section{RGB Image Transport} \label{sec:RGB}
In our first experiment we consider the periodic color OT between two RGB images $u_0$ and $u_1$ that are used as densities $f^0$ and $f^1$ respectively. 
At this point it is important to choose image pairs which have approximately the same mass (i.e.\ the overall sum of all pixels and color channels) 
as one needs to rescale the images such that $\norm{f^0}{1} = \norm{f^1}{1}$.
The results of four different examples are shown in Fig.~\ref{Fig:RGB_transport}. 
The first row shows an artificial example of the transport between one red Gaussian into a blue and a green Gaussian with smaller variance. 
In the second row, two polar lights of different color and shape are transported into each other. 
The third row illustrates how a topographic map of Europe is transported into a satellite image of Europe at night. 
Finally, the last row displays the transport between two cranes in Hamburg. 
All images are of size $100\times 100\times 3$ and in each case, we used $P = 32$ time steps and $2000$ iterations in our algorithm. 
Note that, however, already after $200$ iterations there are no visible changes any longer.
The images are depicted at intermediate times $t = \frac{i}{8}, i= 0,\dots,8$. 
In all cases one nicely sees a continuous change of color and shape during the transport.
\section{Hue Histogram Transport} \label{sec:hue}
\begin{figure} \centering
{\includegraphics[width=.10\textwidth]{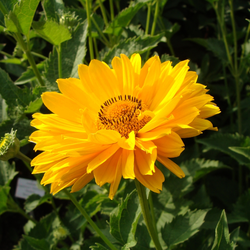}} 
{\includegraphics[width=.10\textwidth]{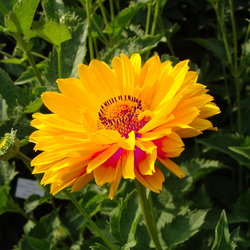}} 
{\includegraphics[width=.10\textwidth]{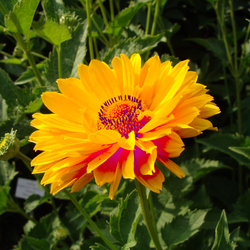}} 
{\includegraphics[width=.10\textwidth]{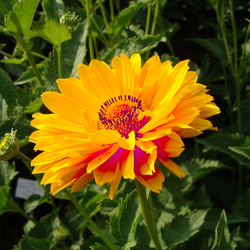}} 
{\includegraphics[width=.10\textwidth]{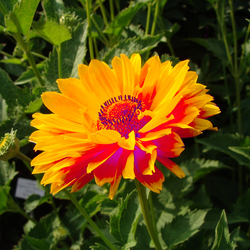}} 
{\includegraphics[width=.10\textwidth]{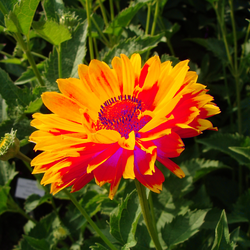}} 
{\includegraphics[width=.10\textwidth]{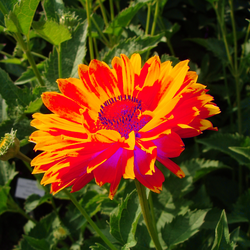}} 
{\includegraphics[width=.10\textwidth]{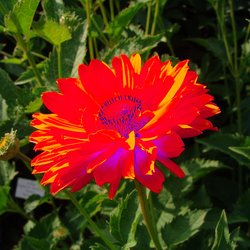}} 
{\includegraphics[width=.10\textwidth]{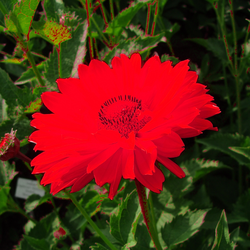}} 
\caption{\label{Fig:hue_transport}
OT between hue histograms and histogram specification at time  $t = \frac{i}{8}, i = 0,\dots, 8$.}
\end{figure}
%
In this section we perform color OT in the HSV space.
We assume the final and target image have the same saturation and value such that only the cyclic 
hue component has to be transported: Assume we are given two images $u_i$, $i=0,1$ 
which differ only in the hue component represented by their normalized histograms $h^i$  as empirical densities $f^i$, $i=0,1$. 
As the hue component is periodic, this fits into our setting. 
The intermediate histograms $h^t$, $t \in (0,1)$ are then used to obtain the hue images via histogram specification. 
Together with the original saturation and intensity they yield the images $u_t$, $t \in (0,1)$. 
For the histogram specification of periodic data we have applied the analysis in~\cite{RDG11} 
and the exact histogram specification method for real-valued data 
proposed, e.g., in~\cite{NS14}.
Fig.~\ref{Fig:hue_transport} shows an example, where the histogram of the hue component of a yellow flower is transported into the one of a red flower. 
The color changes gradually and in a realistic way which would not be the case if the periodicity of the hue histogram is not taken into account.
\\
{
	{\bf Acknowledgement:}
	Funding by the DFG within the Research Training Group 1932 is gratefully acknowledged.}
\footnotetext[1]{
	All images from Wikimedia Commons: AGOModra\_aurora.jpg by Comenius University under CC BY SA 3.0, Aurora-borealis\_andoya.jpg by M.~Buschmann under CC BY 3.0, Europe\_satellite\_orthographic.jpg and Earthlights\_2002.jpg by NASA, K\"ohlbrandbr\"ucke5478.jpg by G.~Ries under CC BY SA 2.5, K\"ohlbrandbr\"ucke.jpg by HafenCity1 under CC BY 3.0.}
\bibliographystyle{abbrv}

\end{document}